\documentclass[11pt,a4paper]{amsart}
\usepackage{amsthm,amsmath,amscd}
\title{CI-sequences and almost complete intersections}
\author{Giuseppe Zappal\`a}

\subjclass[2010]{13 H 10, 14 N 20, 13 D 40}

\DeclareSymbolFont{rsfscript}{OMS}{rsfs}{m}{n}
\DeclareSymbolFontAlphabet{\mathrsfs}{rsfscript}

\DeclareSymbolFont{AMSb}{U}{msb}{m}{n}
\DeclareSymbolFontAlphabet{\mathbb}{AMSb}

\DeclareSymbolFont{eufrak}{U}{euf}{m}{n}
\DeclareSymbolFontAlphabet{\gothic}{eufrak}

\newcommand{\Alt}{\operatorname{Alt}}

\newcommand{\pp}{\mathbb P}
\newcommand{\af}{\mathbb A}

\newcommand\CI{\operatorname{CI}}
\newcommand\ACI{\operatorname{ACI}}

\newcommand\rank{\operatorname{rank}}

\newcommand\Gor{\operatorname{Gor}}

\newcommand\zz{{\mathbb Z}}

\newcommand{\vt}{\vartheta}

\newtheorem{thm}{Theorem}[section]

\newtheorem{prp}[thm]{Proposition}
\newtheorem{cor}[thm]{Corollary}

\theoremstyle{definition}

\newtheorem{dfn}[thm]{Definition}

\theoremstyle{remark}
\newtheorem{rem}[thm]{Remark}
\newtheorem{exm}[thm]{Example}

\newcounter{num}

\begin{document}



\begin{abstract}
We study the Hilbert function and the graded Betti numbers of almost complete intersection artinian algebras. We show that that every Hilbert function of a complete intersection artinian algebra  is the Hilbert function of an almost complete intersection algebra. In codimension $3$ we focus on almost complete intersection artinian algebras whose Hilbert function coincides with that of a complete intersection defined by $3$ forms of the same degree.
We classify all the possible graded Betti numbers of such algebras and we specify what cancellations are allowed in a minimal graded free resolution.
\end{abstract}

\maketitle

\section{Introduction}
\markboth{\it Introduction}{\it Introduction}
A well-studied and important numerical invariant of a standard graded algebra
is the Hilbert function. It gives the dimension of the graded components of the algebra. Macaulay in \cite{Mac} characterized the numerical sequences that occur as the Hilbert function of some standard graded algebra. 
\par
On the other hand a method to study the structure of a standard graded algebra is to find its graded free resolution. It determines the graded Betti numbers of the algebra, which are a more refined numerical invariant. Indeed the graded Betti numbers of the algebra determine the Hilbert function. 
\par
Thus a classical problem is to compute all the possible graded Betti numbers compatible with an assigned Hilbert function. When $H$ is the Hilbert function of a Cohen-Macaulay standard graded algebra of codimension $2$ this problem was solved in \cite{Ca}, but the situation becomes more and more complicated when we treat codimensions greater than $2.$ In this case some sporadic result is available. 
For instance, in codimension $3,$ all the possible graded Betti numbers were found in \cite{CV} and \cite{RZ2}, for Gorenstein artinian algebras with an assigned Hilbert function. 
\par
In this paper we focus our attention on $\CI$-sequences (i.e Hilbert functions of complete intersection artinian algebras) and on almost complete intersection artinian algebras (i.e. artinian algebras $R/I$ of codimension $c,$ where $R$ is a polynomial ring in $c$ variables and $I$ is an ideal with exactly $c+1$ minimal generators).
\par
Almost complete intersection ideals were extensively studied since they have one more generator with respect a complete intersection ideal with the same codimension and also since they are directly linked to Gorenstein ideals. In particular, in codimension $3,$ using the remarkable structure theorem of Buchsbaum and Eisenbud (see \cite{BE}), the resolution and the graded Betti numbers of almost complete intersection artinian algebras are well understood (see \cite{RZ1}, \cite{RZ4}and \cite{Z}).
\par
In section $3$ we show that every $\CI$-sequence is the Hilbert function of an almost complete intersection algebra, which we can actually take generated by monomials.
\par
In section $4,$ in Theorems \ref{t1}, \ref{t2}, \ref{t3}, we find all the possible graded Betti numbers of the almost complete intersection artinian algebras whose Hilbert function is $H_{\CI(a,a,a)}$ i.e. the Hilbert function of a standard graded complete intersection artinian algebra defined by an ideal minimally generated by $3$ forms of degree $a.$ We need to use a careful liaison argument to guarantee that the linked Gorenstein ideal contain a regular sequence of prescribed type. We specify also what types of consecutive cancellations are allowed in the resolutions (see theorems quoted above and Corollary \ref{ahcanc}). Moreover we compute the maximum number of minimal second syzygies admissible for such algebras (see Corollary \ref{tmax}).

\section{Preliminary facts}
\markboth{\it Preliminary facts}
{\it Preliminary facts}
Throughout the paper $k$ will be a infinite field and $R:=k[x_1,\ldots,x_c],$ $c\ge 3,$ will be the standard graded polynomial $k$-algebra. We consider only homogeneous ideals.
\par
Let $A=R/I$ be a standard graded artinian $k$-algebra, where $I$ is a homogeneous ideal of $R.$ Recall that $A$ is said to be an
{\em almost complete intersection algebra} if $I$ is minimally generated by $c+1$ elements (in this case $I$ is said to be an almost complete intersection ideal). Every almost complete intersection ideal is directly linked to a Gorenstein ideal in a complete intersection.
\par
Gorenstein ideals of codimension $3$ are well understood, because of the structure theorem of Buchsbaum and Eisenbud (see \cite{BE}).
The possible graded Betti numbers for Gorenstein ideals of codimension $3$ are described in the following theorem.
\begin{thm}\label{gaeta}
Let $(d_1,\ldots, d_{2n+1}),$ $d_1 \le\ldots\le d_{2n+1},$ be a sequence of positive integers. It is the sequence of the degrees of the minimal generators of a Gorenstein ideal of codimension $3$ iff
\begin{itemize}
\item [1)] $\vt=\frac{1}{n} \sum_{i=1}^{2n+1} d_i$ is an integer;
\item [2)] $\vt > d_i + d_{2n+3-i}$ for $2 \le i \le n$ (Gaeta conditions).
\end{itemize}
\end{thm}
\begin{proof}
This is the main result in \cite{Di}.
\end{proof}
Moreover a description of the possible graded Betti numbers for Gorenstein ideals of codimension $3$ relatively to the Hilbert function can be found in \cite{RZ2}.
\par
Let $\delta=(d_1,\ldots, d_{2n+1}),$ $d_1 \le\ldots\le d_{2n+1},$ be the sequence of the degrees of the minimal generators of a Gorenstein ideal of codimension $3.$
\par
We set again $\vt=\frac{1}{n} \sum_{i=1}^{2n+1} d_i.$ Then we can define an alternating matrix 
$\Alt(\delta)=(a_{ij}),$ in the ring $S=k[\{x_{ij}\}],$ in the following way
$$a_{ij}=\begin{cases}
  x_{ij}^{\vt-d_i-d_j} & \text{ if $i<j$ and $\vt-d_i-d_j>0$} \\
  0 & \text{ if $i<j$ and $\vt-d_i-d_j \le 0$}.
\end{cases}$$
Let $I \subset S$ be the ideal generated by the submaximal pfaffians of $A.$ Then an artinian reduction of  $S/I$ is a graded Gorenstein artinian algebra with sequence of the degrees of the minimal generators $\delta.$ 
\par
Let $\Gor(\delta)$ be the set of all Gorenstein ideals $I \subset k[x_1,x_2,x_3],$ whose
sequence of the degrees of the minimal generators is $\delta.$
 
Moreover we define
\begin{multline*}
  \CI_{\delta}=\{(a_1,a_2,a_3)\in(\zz^{+})^3\mid a_1 \le a_2 \le a_3 \mbox{ and }\exists I \in \Gor(\delta) \mbox{ containing} \\
	\mbox{a regular sequence of degrees } (a_1,a_2,a_3)\}.
\end{multline*}
To study ideals linked to Gorenstein ideals in a complete intersection, in codimension $3,$ we need the following result.

\begin{thm}\label{c3b}
The set $\CI_{\delta}$ has a unique minimal element.
\end{thm}
\begin{proof}
See Theorem 3.6 in \cite{RZ3}, where this minimal element is explicitly computed.
\end{proof}
We will denote the only minimal element of $\CI_{\delta}$ with $\min(\delta).$
\par
Let $H$ be the Hilbert function of a standard graded artinian algebra. A classical problem in Algebraic Geometry is to classify all the possible graded Betti numbers of the standard graded artinian algebras whose Hilbert function is $H.$ This problem is completely settled in codimension $2$ (see \cite{Ca}). It is an open problem, in general, in codimension greater than two.

\begin{dfn}\label{canc}
Let $I \subset k[x_1,\ldots,x_c]$ be a homogeneous ideal and let $F_{\bullet}$ be a graded minimal free resolution of $R/I.$ 
\par
Let us suppose that there is $i$ such that $F_i=L_i \oplus R(-a)$ and 
$F_{i+1}=L_{i+1} \oplus R(-a).$ We say that the module $R(-a)$ is {\em cancellable} if there exists a homogeneous ideal $J$ such that the graded minimal free resolution of $R/J$ is $G_{\bullet},$ 
where $G_i=L_i,$ $G_{i+1}=L_{i+1}$ and $G_n=F_n$ otherwise. 
\par
Let us suppose that there is $i$ such that $F_i=L_i \oplus R(-a) \oplus R(-b)$ and 
$F_{i+1}=L_{i+1} \oplus R(-a) \oplus R(-b).$ We say that the module $R(-a) \oplus R(-b)$ is a 
{\em cancellable couple} if there exists a homogeneous ideal $J$ such that the graded minimal free resolution of $R/J$ is $G_{\bullet},$ where $G_i=L_i,$ $G_{i+1}=L_{i+1}$ and $G_n=F_n$ otherwise.
\end{dfn}

\begin{rem}
Note that, in the situation of Definition \ref{canc}, $R/I$ and $R/J$ have the same Hilbert function. 
By \cite{Pa}, $I$ and $J$ are connected by a sequence of deformations over $\af^1_{k}.$
So it is very important to study what modules in a graded minimal free resolution are cancellable.
See also \cite{Pe} for a discussion about cancellations in graded free resolutions.
\end{rem}

\section{Almost Complete Intersection sequences}
\markboth{\it Almost Complete Intersection sequences}{\it Almost Complete Intersection sequences}

In this section we show that if a sequence occurs as Hilbert function of a complete intersection artinian algebra then it is also the Hilbert function of a monomial almost complete intersection artinian algebra.

\begin{dfn}\label{cise}
A sequence of positive integers $H$ is called a {\em $\CI$-sequence} if it is the Hilbert function of standard graded artinian complete intersection algebra. $H$ is called an {\em $\ACI$-sequence} if it is the Hilbert function of standard graded artinian almost complete intersection algebra.
\end{dfn}
Since the Koszul complex is a graded minimal free resolution for the standard graded artinian complete intersection algebras, it is easy to describe the $\CI$-sequences.

\begin{prp}
In any codimension, every $\CI$-sequence is an $\ACI$-sequence.
\end{prp}
\begin{proof}
Let $H=H_{\CI(a_1,\ldots,a_r)}$ be a $\CI$-sequence, with $2\le a_1\le\ldots\le a_r.$
Let $a_r+1 \le h \le a_r+a_1-1$ and let us consider the ideal
  $$I_Q=(x_1^{a_1},\ldots,x_{r-1}^{a_{r-1}},
  x_r^h,x_1^{a_1+a_r-h}x_2^{a_2-a_1}\cdots x_{r-1}^{a_{r-1}-a_{r-2}}x_r^{h-a_{r-1}}).$$
We would like to show that $H_Q=H.$ 
\par
We set $I_Z=(x_1^{a_1},\ldots,x_{r-1}^{a_{r-1}},x_r^h).$
It is the ideal of a complete intersection and $I_Z\subset I_Q.$ Let $I_G=I_Z:I_Q.$ We have
  $$I_G=(x_1^{h-a_r},x_2^{a_1},\ldots,x_{r-1}^{a_{r-2}},x_r^{a_{r-1}}),$$
hence $I_G$ is the ideal of a complete intersection of type $(h-a_r,a_1,\ldots,a_{r-1}).$ 
Therefore, if we set $\vartheta_Z=h+\sum_{i=1}^{r-1}a_i,$ 
we have
  $$H_Q(n)=H_Z(n)-H_G(\vartheta_Z-n).$$ 
The Hilbert function of $R/I_G$ is equal to the Hilbert function of $R/I_{G'},$ where
  $$I_{G'}=(x_1^{a_1},\ldots,x_{r-1}^{a_{r-1}},x_r^{h-a_r}).$$
But $I_{Z}\subset I_{G'}$ and we set $I_W=I_Z:I_{G'}.$ Of course $I_W=(x_1^{a_1},\ldots,x_{r-1}^{a_{r-1}},x_r^{a_r}),$ hence
 $$H_W(n)=H_Z(n)-H_{G'}(\vartheta_Z-n)=H_Z(n)-H_{G}(\vartheta_Z-n)=H_Q(n),$$ 
so $H_Q=H_W=H_{\CI(a_1,\ldots,a_r)}.$
\end{proof}

\begin{rem}
Note that there are Gorenstein sequences that are not $ACI$-sequences. For instance $H=(1,3,1)$ is a Gorenstein sequence, but it is not an $ACI$-sequence. Indeed if a standard graded artinian algebra $R/I$ has Hilbert function equal to $H,$ then $I$ is forced to have $5$ minimal generators in degree $2.$
\end{rem}

\section{Graded Betti numbers}
\markboth{\it Graded Betti numbers}{\it Graded Betti numbers}

In this section we compute all the possible graded Betti numbers of the almost complete intersection artinian algebras whose Hilbert function is $H_{\CI(a,a,a)}.$
\par
If $A_Q$ is an artinian algebra, we denote by $t_Q$ the number of the last minimal syzygies of $A_Q.$
\par
Let $A_Q$ be an almost complete intersection artinian algebra with Hilbert function $H_{\CI(a,a,a)},$ $a \ge 2.$ Then $I_Q$ has $4$ generators of degrees $a,a,a,h,$ with $a+1 \le h \le 3a-2.$ 
Moreover, since $\Delta^3H_Q(2a)=3,$ $I_Q$ has at least $3$ minimal first syzygies of degree $2a$,
so a graded minimal free resolution of $A_Q$ is of the type
 $$0 \to G_3 \oplus R(-3a) \to G_3 \oplus R(-2a)^3 \oplus R(-h)  \to R(-a)^3 \oplus R(-h) \to R, $$
with $\rank G_3=t_Q-1.$
\par
Let $A_Q=R/I_Q$ be an almost complete intersection artinian algebra of codimension $3.$ Then there is a fundamental numerical invariant $d^*,$ (which is the degree of one of the $4$ minimal generators of $I_Q$) that can be computed directly by the graded Betti numbers of $A_Q$ (see Theorem 4.1 and Proposition 4.3 in \cite{Z}). We will compute $d^*,$ in our framework, in Corollary \ref{dstar}.
\begin{prp}\label{aph}
Let $A_Q=R/I_Q$ be an almost complete intersection artinian algebra with Hilbert function $H_{\CI(a,a,a)}.$ Then $A_Q$ has a minimal first syzygy of degree $a+h$ iff $t$ is even.
\end{prp}
\begin{proof}
We can perform a liaison of $I_Q$ in a complete intersection ideal $I_Z$ of type $(a,a,h).$ Then $I_G=I_Z:I_Q$ is a Gorenstein ideal, such that $H_G=H_{\CI(h-a,a,a)}.$ We set $I_Z=(f_a,g_a,f_h),$ with $\deg f_a=\deg g_a=a$ and $\deg f_h=h.$
\par
Let us suppose that $A_Q$ has a minimal first syzygy of degree $a+h.$ 
Then  $a+h \le 3a-1,$ hence $a+1 \le h \le 2a-1.$ Consequently $A_Q$ has a resolution of type
\begin{multline*}
 0 \to G'_3 \oplus R(-(a+h)) \oplus R(-3a) \to \\ 
\to G'_3 \oplus R(-2a)^3 \oplus R(-(a+h))\oplus R(-h) \to R(-a)^3 \oplus R(-h) \to R,
\end{multline*}
with $\rank G'_3=t-2.$
Using mapping cone, we can get a resolution (not necessarily minimal) of $I_G.$ In particular, since 
$\vartheta_Z=2a+h,$ we obtain for $I_G$ generators of degrees $h-a,a,a,a,h$ and other $t-2$ generators of degrees greater than $a.$ Among these generators, the only ones that can eventually be non-minimal for $I_G$ are $f_a,g_a,f_h.$
But $H_G=H_{\CI(h-a,a,a)},$ so $\vartheta_G=a+h;$ therefore $f_h$ is not minimal for $I_G,$ since, otherwise, $I_G$ should have a minimal first syzygy of degree $a.$ This is impossible since $I_G$ has only one minimal generator of degree smaller than $a.$ 
\par
Consequently $I_G$ has one generator of degree $h-a,$ three generators of degree $a$ and other $t-2$ generators of degrees greater than $a.$ But, again, $H_G=H_{\CI(h-a,a,a)},$ so $I_G$ has exactly $2$ minimal generators of degree $a,$ therefore $I_G$ has $t+3$ minimal generators. Consequently $t+3$ is an odd number, i.e. $t$ is even.
\par
If $t$ is even and $d^*=a,$ then, for the structure of first syzygies of an almost complete intersection artinian algebra of codimension $3,$ $I_Q$ has minimal first syzygies of degrees $a+h,2a,2a.$ 
\par
If $t$ is even and $d^*=h,$ then, for the same reason, $I_Q$ has three minimal first syzygies of degrees $a+h.$ 
\end{proof}

\begin{prp}\label{apiuh}
Let $A_Q=R/I_Q$ be an almost complete intersection artinian algebra with Hilbert function $H_{\CI(a,a,a)}.$
Let us suppose that $t$ is even.
Then $A_Q$ has only one minimal first syzygy of degree $a+h.$
\end{prp}
\begin{proof}
Let $r$ be the number of minimal first syzygies of $I_Q$ of degree $a+h.$ 
By Proposition \ref{aph}, $u \ge 1.$ 
Moreover $I_Q$ has exactly $u$ minimal second syzygies of degree $a+h.$ 
\par
We link again $I_Q$ in a complete intersection ideal $I_Z$ of type $a,a,h.$ Let $I_G=I_Z:I_Q.$ $I_G$ has only two minimal generators of degree $a.$ Since $t$ is even, only one of the three generators of $I_Z$ is minimal for $I_G,$ so $I_G$ has $u+1=2$ minimal generators of degree $a.$ Hence $u=1.$
\end{proof}

\begin{cor}\label{dstar}
Let $A_Q=R/I_Q$ be an almost complete intersection artinian algebra with Hilbert function $H_{\CI(a,a,a)}.$
Let $(a,a,a,h)$ be the vector of the minimal generators degrees of $I_Q,$ $a+1 \le h \le 3a-2.$
\begin{itemize}
\item[1)] If $t_Q$ is even then $d^*=a.$
\item[2)] If $t_Q$ is odd then $d^*=h.$
\end{itemize}
\end{cor}
\begin{proof}
If $t_Q$ is even then, by Proposition \ref{apiuh}, $A_Q$ has only one minimal first syzygy of degree $a+h.$ Hence $d^*\ne h,$ i.e. $d^*=a.$
\par
If $t_Q$ is odd then $A_Q$ has not first syzygies of degree $a+h$ by Proposition \ref{aph}. 
So we get $d^*=h.$
\end{proof}

\begin{cor}\label{todd}
Let $A_Q=R/I_Q$ be an almost complete intersection artinian algebra with Hilbert function $H_{\CI(a,a,a)}.$
Let $(a,a,a,h)$ be the vector of the minimal generators degrees of $I_Q,$ $a+1 \le h \le 3a-2.$
\par
If $h \ge 2a$ then $t_Q$ is odd.
\end{cor}
\begin{proof}
If $t_Q$ is even then by Proposition \ref{aph} $A_Q$ has a minimal first syzygy of degree $a+h.$ 
The highest second syzygy of $A_Q$ has degree $3a,$ so $a+h \le 3a-1,$ hence $a+1 \le h \le 2a-1.$
\end{proof}

\begin{prp}
Let $A_Q=R/I_Q$ be an almost complete intersection artinian algebra with Hilbert function $H_{\CI(a,a,a)}.$ 
Let $(a,a,a,a+1)$ be the vector of the minimal generators degrees of $I_Q.$
\par
Then $t_Q=2$ and the graded minimal free resolution of $A_Q$ is
  \begin{multline*}
 0 \to R(-(2a+1)) \oplus R(-3a) \to R(-2a)^3 \oplus R(-(2a+1))\oplus R(-(a+1)) \to \\ 
 \to R(-a)^3 \oplus R(-(a+1)) \to R.
  \end{multline*}
An example of this is the monomial ideal $I_Q=(x^a,y^{a+1},z^a,x^{a-1}y),$ in $k[x,y,z].$
\end{prp}
\begin{proof}
We link $I_Q$ in a complete intersection ideal $I_Z$ of type $(a,a,a+1).$ Let $I_G=I_Z:I_Q.$
It is a Gorenstein ideal with Hilbert function $H_{\CI(1,a,a)},$ so $I_G$ is a complete intersection 
ideal of type $(1,a,a).$ So we have a liaison of a complete intersection $\CI(1,a,a)$ in a complete intersection $\CI(a,a,a+1).$
\par
We have two possibilities. In the liaison we can use two minimal generators of degree $a$ or one minimal and one not minimal generator of degree $a.$ In the first case we get a complete intersection $\CI(a,a,a).$ In the second case we get an almost complete intersection whose grade minimal free resolution is the one requested.
\end{proof}

According to what we said, if $t=t_Q$ is even, the graded minimal free resolution of $A_Q$ is of the type
 \begin{multline*}
 0 \to G \oplus R(-(a+h)) \oplus R(-3a) \to \\ 
\to G \oplus R(-2a)^3 \oplus R(-(a+h))\oplus R(-h) \to R(-a)^3 \oplus R(-h) \to R
\end{multline*}
$a \ge 2,$ $a+1 \le h \le 2a-1,$
where $G$ is a graded free module, $\rank G=t-2,$ $G \cong G^{\vee}(-d),$ and $d=3a+h.$
Therefore we have a decomposition $G \cong L \oplus L^{\vee}(-d),$ where $L$ is a graded free module, 
$\rank L=\frac{t-2}{2}.$ So we obtain the graded minimal free resolution
\begin{multline*}
 0 \to L \oplus L^{\vee}(-(3a+h))  \oplus R(-(a+h)) \oplus R(-3a) \to \\ 
\to L \oplus L^{\vee}(-(3a+h)) \oplus R(-2a)^3 \oplus R(-(a+h))\oplus R(-h) \to \\ \to R(-a)^3 \oplus R(-h) \to R.
\end{multline*}
Analogously, if $t=t_Q$ is odd, the graded minimal free resolution of $A_Q$ is of the type
\begin{multline*}
 0 \to L \oplus L^{\vee}(-(3a+h))  \oplus R(-3a) \to \\ 
\to L \oplus L^{\vee}(-(3a+h)) \oplus R(-2a)^3 \oplus R(-h) \to R(-a)^3 \oplus R(-h) \to R
\end{multline*}
$a \ge 2,$ $a+2 \le h \le 3a-2,$ $\rank L=\frac{t-1}{2}.$
\par
Now we prove that if we fix $H=H_{\CI(a,a,a)},$ $h$ and the parity of $t_Q$ there exists a maximal Betti sequence $\beta$ of almost complete intersection artinian algebras compatible with this data and all the other Betti sequences can be obtained from $\beta$ carrying out suitable cancellations. We will see that not all cancellations are allowed and we will specify which are.

\begin{thm}\label{t1}
Let $H=H_{\CI(a,a,a)}.$ Let $a+1 \le h \le 2a-1.$ We set
$$F=
\begin{cases}
R(-(2a+1)) \oplus \ldots \oplus R(-(a+h)) & \text{ if $h-a$ is odd} \\
R(-(2a+1)) \oplus \ldots \oplus R(-(a+h)) \oplus R(-(\frac{3a+h}{2})) & \text{ if $h-a$ is even}
\end{cases}$$
Then there exists an almost complete intersection artinian algebra with Hilbert function $H$ and graded minimal free resolution 
\begin{multline*}
 0 \to F \oplus R(-3a) \to  
 R(-h) \oplus R(-2a)^3 \oplus F \to 
 R(-a)^3 \oplus R(-h) \to R
\end{multline*}
Moreover let ${\mathcal A}_{H,h}$ be the set of the almost complete intersection artinian algebras with Hilbert function $H,$ generated in degrees $(a,a,a,h),$ $a+1 \le h \le 2a-1,$ and $t$ even. Then
\begin{itemize}
 \item[1)] these graded Betti numbers are maximal among algebras in ${\mathcal A}_{H,h};$
 \item[2)] the only couples of cancellations allowed, among algebras in ${\mathcal A}_{H,h},$ are  
 $R(-i) \oplus R(-(3a+h-i)),$ for $2a+1 \le i \le a+h-1.$
\end{itemize}
\end{thm}
\begin{proof}
We set
$$\delta=
\begin{cases} (h-a,a,a,a+1,\ldots,h-1) & \text{ if $h-a$ is odd } \\ 
\left(h-a,a,a,a+1,\ldots,\frac{a+h}{2},\frac{a+h}{2},\ldots,h-1\right) & \text{ if $h-a$ is even }
\end{cases}$$
By Theorem \ref{gaeta} and by Theorem 3.6 in \cite{RZ3}, there exists a Gorenstein ideal $I_G$ with sequence of the degrees of the minimal generators $\delta$ and containing a complete intersection ideal $I_Z$ of type $(a,a,h),$ with a generator of degree $a$ minimal and a generator of degree $a$ not minimal
for $I_G.$
The linked ideal $I_Q=I_Z:I_G$ has the requested graded minimal free resolution.
\par
Let $A_Q=R/I_Q$ be an almost complete intersection artinian algebra with Hilbert function 
 $H,$ $a+1 \le h \le 2a-1,$ and $t$ even. We can link $I_Q$ in a complete intersection ideal $I_Z$ of type $(a,a,h).$ Let $I_G=I_Z:I_Q.$ Then $A_G=R/I_G$ is a Gorenstein algebra such that $H_G=H_{\CI(h-a,a,a)}.$  We recall that the resolution of $A_G$ is determined by the degrees of the minimal generators of $I_G.$ The maximal Betti numbers relatively to the Hilbert function were computed in \cite{RZ2}, Proposition 3.7 and Remark 3.8. Since $\Delta^2H_{\CI(h-a,a,a)}(n)=-2,$ for $a \le n \le h-1,$ they are given by the resolution
  $$\ldots \to R(-(h-a))\oplus R(-a)^2 \oplus R(-(a+1))^2 \oplus \ldots \oplus R(-(h-1))^2 \to R.$$
But if $I_G$ has two minimal generators in degree $n,$ $a+1 \le n \le \left\lfloor\frac{a+h}{2} \right\rfloor,$ by Theorem 4.4 in \cite{RZ5} or also by Theorem 3.6 in \cite{RZ3}, it cannot contain a regular sequence of type $(a,a,h).$ Therefore the maximal sequence of the degrees of the minimal generators for a Gorenstein ideal $I_G,$ with $H_G=H_{\CI(h-a,a,a)}$ and 
containing a complete intersection ideal $I_Z$ of type $(a,a,h)$ is $\delta.$
\par
The cancellations allowed are the ones that depend from cancellations of the linked Gorenstein. So it is enough to use again Proposition 3.7 and Remark 3.8 in \cite{RZ2}.
 \end{proof}

\begin{thm}\label{t2}
Let $H=H_{\CI(a,a,a)}.$ Let $a+2 \le h \le 2a-1.$ We set
$$F=
\begin{cases}
R(-(2a+1)) \oplus \ldots \oplus R(-(a+h-1)) & \text{ if $h-a$ is odd} \\
R(-(2a+1)) \oplus \ldots \oplus R(-(a+h-1)) \oplus R(-(\frac{3a+h}{2})) & \text{ if $h-a$ is even}
\end{cases}$$
Then there exists an almost complete intersection artinian algebra with Hilbert function $H$ and graded minimal free resolution 
\begin{multline*}
 0 \to F \oplus R(-3a) \to  
 R(-h) \oplus R(-2a)^3 \oplus F \to 
 R(-a)^3 \oplus R(-h) \to R
\end{multline*}
Moreover let ${\mathcal B}_{H,h}$ be the set of the almost complete intersection artinian algebras with Hilbert function 
$H,$ generated in degrees $(a,a,a,h),$ $a+2 \le h \le 2a-1,$ and $t$ odd. Then
\begin{itemize}
 \item[1)] these graded Betti numbers are maximal among algebras in ${\mathcal B}_{H,h};$
 \item[2)] the only couples of cancellations allowed, among algebras in ${\mathcal B}_{H,h},$ are  
 $R(-i) \oplus R(-(3a+h-i)),$ for $2a+1 \le i \le a+h-1$ and $t\ge 5.$
\end{itemize}
\end{thm}
\begin{proof}
We set $\delta$ as in the proof of Theorem \ref{t1}, but now we link using a complete intersection ideal $I_Z$ of type $(a,a,h),$ with two generators of degree $a$ both minimal for $I_G.$ 
\par
The rest of the proof runs analogously as in in the proof of Theorem \ref{t1}.
\end{proof}

\begin{thm}\label{t3}
Let $H=H_{\CI(a,a,a)}.$ Let $2a \le h \le 3a-2.$ We set
$$F=
\begin{cases}
R(-(h+1)) \oplus \ldots \oplus R(-(3a-1)) & \text{ if $h-a$ is odd} \\
R(-(h+1)) \oplus \ldots \oplus R(-(3a-1)) \oplus R(-(\frac{3a+h}{2})) & \text{ if $h-a$ is even}
\end{cases}$$
Then there exists an almost complete intersection artinian algebra with Hilbert function $H$ and graded minimal free resolution 
\begin{multline*}
 0 \to F \oplus R(-3a) \to  
 R(-h) \oplus R(-2a)^3 \oplus F \to 
 R(-a)^3 \oplus R(-h) \to R
\end{multline*}
Moreover let ${\mathcal B}_{H,h}$ be the set of the almost complete intersection artinian algebras with Hilbert function $H,$ generated in degrees $(a,a,a,h),$ $2a \le h \le 3a-2$ (in this case $t$ is necessarily odd). Then
\begin{itemize}
 \item[1)] these graded Betti numbers are maximal among algebras in ${\mathcal B}_{H,h};$
 \item[2)] the only couples of cancellations allowed, among algebras in ${\mathcal B}_{H,h},$ are  
 $R(-i) \oplus R(-(3a+h-i)),$ for $h+1 \le i \le 3a-1$ and $t \ge 5.$
 \end{itemize}
\end{thm}
\begin{proof}
We set
$$\delta=
\begin{cases} (a,a,h-a,h-a+1,\ldots,2a-1) & \text{ if $h-a$ is odd } \\ 
\left(a,a,h-a,h-a+1,\ldots,\frac{a+h}{2},\frac{a+h}{2},\ldots,2a-1\right) & \text{ if $h-a$ is even }
\end{cases}$$
Again we link using a complete intersection ideal $I_Z$ of type $(a,a,h),$ with two generators of degree $a$ both minimal for $I_G$ (in this case we are forced for the minimality). 
\par
The rest of the proof runs analogously as in the proof of Theorem \ref{t1}.
\end{proof}

\begin{cor}\label{tmax}
The maximum number of minimal second syzygies for an almost complete intersection artinian algebra with Hilbert function $H_{\CI(a,a,a)}$ is $t^{\max}_a=a+1$ if $a$ is even, $t^{\max}_a=a$ if $a$ is odd.
\end{cor}
\begin{proof}
Comparing the maximal graded Betti numbers that we state in Theorem \ref{t1}, Theorem \ref{t2} and
Theorem \ref{t3}, we get this maximum number when $h=2a.$
\end{proof}

\begin{cor}\label{ahcanc}
Let $A_Q=R/I_Q$ be an almost complete intersection artinian algebra with Hilbert function $H_{\CI(a,a,a)},$ with $I_Q$ generated in degrees $a,a,a,h.$ 
Let us suppose that $A_Q$ has a minimal first syzygy of degree $a+h.$ 
\par
Then the module $R(-(a+h))$ is cancellable in the minimal graded free resolution of $A_Q$ iff $t_Q \ge 4.$
\end{cor}
\begin{proof}
By Proposition \ref{aph}, $t_Q$ is even and by Proposition \ref{apiuh}, $A_Q$ has only one minimal first syzygy and only one minimal second syzygy of degree $a+h.$ By Corollary \ref{todd}, $a+1 \le h \le 2a-1.$
\par
If $R(-(a+h))$ is cancellable then there exists an almost complete intersection artinian algebra $A_W$ with 
$t_W=t_Q-1;$ since $t_W$ is odd and $A_W$ is not a complete intersection algebra, $t_W \ge 3$ i. e. $t_Q \ge 4.$ 
\par
If $t_Q \ge 4,$ then by Theorem \ref{t2} there exists an almost complete intersection artinian algebra with the same graded minimal free resolution of $A_Q$ except for the module $R(-(a+h)),$ that is therefore cancellable.
\end{proof}

\begin{exm}
Let $a=3$ and $h=5.$ By Theorem \ref{t1}, the maximal resolution $F_{\bullet}$ with these data is
\begin{multline*}
  0 \to R(-7)^2 \oplus R(-8) \oplus R(-9) \to R(-5) \oplus R(-6)^3 \oplus R(-7)^2 \oplus R(-8) \to \\
  \to R(-3)^3 \oplus R(-5) \to R
\end{multline*}
An ideal with this minimal graded free resolution can be realized in the following way. 
We set $\vartheta_Z=2a+h=11.$ Let $v=(7,7,8,9)$ be the vector of the degrees of the minimal second syzygies. 
We compute $v'=(11-9,11-8,11-7,11-7)=(2,3,4,4).$ We extend $v'$ with a component equal to $a.$ So we obtain the vector $u=(2,3,3,4,4).$ Now let us consider the alternating matrix $A=\Alt(2,3,3,4,4).$
 $$A=\begin{pmatrix} 
 0              & x_{12}^3  & x_{13}^3   & x_{14}^2 & x_{15}^2 \\ 
 -x_{12}^3 & 0              & x_{23}^2   & x_{24}     & x_{25}     \\
 -x_{13}^3 & -x_{23}^2 & 0               & x_{34}     & x_{35}     \\
 -x_{14}^2 & -x_{24}     & -x_{34}      & 0             & 0 \\
 -x_{15}^2 & -x_{25}     & -x_{35}      & 0             & 0 
 \end{pmatrix}.$$
Let $p_i$ be the pfaffian obtained by $A,$ by deleting the $i$-th row and $i$-th column. 
Let $p_{i,j,k}$ be the pfaffian obtained by $A,$ by deleting the $i$-th, $j$-th and $k$-th row and 
$i$-th, $j$-th and $k$-th column.
\par
Then the ideal $I_Q=(y_2p_1,p_2,y_1p_5,y_1y_2p_{1,2,5})$ is a perfect ideal in the ring $k[\{x_{ij}\},y_1,y_2]$ with graded minimal free resolution $F_{\bullet}.$ So it is enough to consider a artinian reduction of $I_Q.$
The module $R(-8)$ (where $8=a+h$) is cancellable. Indeed the ideal $I_W=(p_2,p_3,y_1p_5,y_1p_{2,3,5})$ has graded minimal free resolution
  \begin{multline*}
  0 \to R(-7)^2 \oplus R(-9) \to R(-5) \oplus R(-6)^3 \oplus R(-7)^2 \to \\
  \to R(-3)^3 \oplus R(-5) \to R.
\end{multline*}
Note that, in the previous resolution, $R(-7)^2$ trivially is not cancellable. Also $R(-7)$ is not cancellable by Theorem \ref{t2}.
So these graded Betti numbers are minimal for almost complete intersection algebras.


\end{exm}

\vspace{1cm}

{\sc Dipartimento di Matematica e Informatica, \\ Universit\`a di Ca\-tania, \\
                  Viale A. Doria 6, 95125 Catania, Italy}\par
{\it E-mail address: }{\tt zappalag@dmi.unict.it}

\end{document}